\documentclass[leqno]{amsart}
\usepackage{caption}
\usepackage{amsmath,amssymb,amsthm,mathtools,dsfont,mathrsfs,tikz}
\usepackage{environ}
\usepackage{stmaryrd}
\usetikzlibrary{arrows}
\usetikzlibrary{positioning}
\usetikzlibrary{decorations.text}
\usetikzlibrary{decorations.pathmorphing}
\usetikzlibrary{decorations.pathreplacing}
\makeatletter
\newsavebox{\measure@tikzpicture}
\NewEnviron{scaletikzpicturetowidth}[1]{%
  \mathrm{Def}\tikz@width{#1}%
  \mathrm{Def}\tikzscale{1}\begin{lrbox}{\measure@tikzpicture}%
  \BODY
  \end{lrbox}%
  \pgfmathparse{#1/\wd\measure@tikzpicture}%
  \edef\tikzscale{\pgfmathresult}%
  \BODY
}
\makeatother
\usepackage{bbm}
\DeclarePairedDelimiter\norm{\lvert}{\rvert}
\DeclarePairedDelimiter\inner{\langle}{\rangle}
\usepackage{bm}
\usepackage[colorlinks=true,linkcolor=blue,citecolor=blue]{hyperref}
\usepackage{enumitem}
\usepackage{csquotes}

\numberwithin{equation}{section}
\newcounter{intro}

		\newtheorem{introthm}[intro]{Theorem}

		\newtheorem{thm}[equation]{Theorem}
		\newtheorem{lem}[equation]{Lemma}
		\newtheorem{cor}[equation]{Corollary}
		\newtheorem{prop}[equation]{Proposition}
		
\theoremstyle{remark}
		\newtheorem{rem}[equation]{Remark}
\theoremstyle{definition}
		
		\newtheorem{exam}[equation]{Example}

\title{Vanishing-off subgroups and supercharacter theory products}
\author{Shawn T. Burkett}
\address{Department of Mathematical Sciences, Kent State University, Kent,
Ohio 44240, U.S.A.} \email{sburket1@kent.edu}
\author{Mark L. Lewis}
\address{Department of Mathematical Sciences, Kent State University, Kent,
Ohio 44240, U.S.A.} \email{mllewis@kent.edu}
\date{\today}
\begin{document}
\maketitle
\begin{abstract}
In this paper, we study the vanishing-off subgroups of supercharacters, and use these to determine several new characterizations of supercharacter theory products. In particular, we give a character theoretic characterization that allows us to conclude that one may determine if a supercharacter theory is a $\Delta$-product or $\ast$-product from the values of its corresponding supercharacters.
\end{abstract}
\section{Introduction}
Let $G$ be a finite group. A supercharacter theory of $G$, a concept formalized by Diaconis and Isaacs in \cite{ID07}, essentially consists of two entities: a distinguished set of mutually orthogonal characters, called supercharacters, and a partition of $G$ into unions of conjugacy classes, called superclasses. The supercharacters and superclasses of a supercharacter theory exhibit much of the same duality as the irreducible characters and conjugacy classes of the group do, and the usual character theory of $G$ is a trivial example of a supercharacter theory, often denoted by $\mathsf{m}(G)$. Every supercharacter theory $\mathsf{S}$ of $G$ gives rise to a subalgebra $\mathrm{scf}(\mathsf{S})$ of the algebra $\mathrm{cf}(G)$ of complex-valued class functions of $G$, and this subalgebra is the $\mathbb{C}$-linear span of the supercharacters of $\mathsf{S}$. Although supercharacter theory was developed primarily to simplify the complicated character theory of certain families of groups, one may not be able to deduce much group theoretic information from a general supercharacter theory. It is reasonable to expect however that the supercharacters of a supercharacter theory $\mathsf{S}$ of $G$ hold some information about the structure of $\mathrm{scf}(\mathsf{S})$.

As with character theory, one pathway for studying the properties of a supercharacter theory is through normal subgroups and quotients. Any normal subgroup which arises as the intersection of kernels of supercharacters of a supercharacter theory $\mathsf{S}$ is called a supernormal subgroup, or $\mathsf{S}$-normal subgroup. Equivalently these subgroups can be expressed as a union of superclasses, also called $\mathsf{S}$-classes. In his Ph.D. thesis \cite{AH08}, Hendrickon introduced the concept of supernormality and developed a way in which one may construct for any $\mathsf{S}$-normal subgroup $N$ a supercharacter theory $\mathsf{S}_N$ of $G$ and a supercharacter theory $\mathsf{S}^{G/N}$ of $G/N$. The $\mathsf{S}_N$-classes are exactly the $\mathsf{S}$-classes that are contained in $N$, and the $\mathsf{S}^{G/N}$ supercharacters can be identified naturally with the $\mathsf{S}$-supercharacters containing $N$ in their kernel. In particular, these supercharacter theories are built from $\mathsf{S}$ in a fairly intuitive way. Hendrickson also showed that the supercharacter theories $\mathsf{S}_N$ and $\mathsf{S}^{G/N}$ may be used to construct a related supercharacter theory, denoted $\mathsf{S}_N\mathop{\ast}\mathsf{S}^{G/N}$ and called the $\ast$-{\it product} of $\mathsf{S}_N$ and $\mathsf{S}^{G/N}$. The supercharacters of this supercharacter theory come in two flavors --- supercharacters of $\mathsf{S}^{G/N}$ naturally considered characters of $G$, and supercharacters of $\mathsf{S}_N$ induced to $G$. The superclasses of $\mathsf{S}_N\mathop{\ast}\mathsf{S}^{G/N}$ are either $\mathsf{S}_N$-superclasses, or preimages of $\mathsf{S}^{G/N}$-superclasses in $G$. It turns out that every $\mathsf{S}_N\mathop{\ast}\mathsf{S}^{G/N}$-superclass is a union of $\mathsf{S}$-superclasses, and that $\mathsf{S}$ coincides with $\mathsf{S}_N\mathop{\ast}\mathsf{S}^{G/N}$ if and only if every $\mathsf{S}$-class lying outside of $N$ is a union of full $N$-cosets \cite[Corollary 3.10]{AH08}. 

In the event that $\mathsf{S}$ is the usual supercharacter theory $\mathsf{m}(G)$, this relates the $\ast$-product to a group theoretic object called a Camina pair. A pair $(G,N)$ is called a {\it Camina pair} if $N$ is a normal subgroup of $G$ for which every conjugacy class of $G$ lying outside of $N$ is a union of full $N$-cosets. It is also known to be equivalent to the condition that every member of $\mathrm{Irr}(G\mid N)$ vanishes off of $N$, where as usual $\mathrm{Irr}(G\mid N)$ dentoes the set of all irreducible characters of $G$ not containing $N$ in their kernel, This connection of Camina pairs to $\ast$-products was explored by Lewis and Wynn in \cite{camina}, where the authors proved that every supercharacter theory of a Camina pair is a $\ast$-product over some normal subgroup of $G$. A natural generalization of both Camina pairs and $\ast$-products can be found in a similar way. Let $M$ and $N$ be $\mathsf{S}$-normal subgroups. Then $\Delta$-product $\mathsf{S}_N\mathop{\Delta}\mathsf{S}^{G/M}$ is the supercharacter theory of $G$ whose superclasses lying outside of $N$ are unions of $M$-classes. In the event that $\mathsf{S}$ is the supercharacter theory $\mathsf{m}(G)$ and coincides with $\mathsf{S}_N\mathop{\Delta}\mathsf{S}^{G/M}$, the tuple $(G,N,M)$ is called a {\it Camina triple}. That is, $(G,N,M)$ is called a Camina triple if for every $g\in G\setminus N$ and $m\in M$, $gm$ is conjugate to $g$. It is precisely this connection between $\Delta$-products and Camina triples that we exploit in this paper.

Camina triples were first studied by by Mattarei in his dissertation (see \cite{SM92}). In \cite{MLvos09}, Lewis studies various properties of {\it generalized Camina pairs}, Camina triples of the form $(G,N,[G,G])$. Many of the results of that paper were later generalized in \cite{NM14} to arbitrary Camina triples by Mlaiki. One of the main tools used by both Lewis and Mlaiki is the so-called {\it vanishing-off} subgroup. The vanishing off subgroup $\mathbf{V}(\chi)$ of an irreducible character $\chi$ of a finite group $G$ is the subgroup generated by all elements $g$ of $G$ satisfying $\chi(g)\neq0$. Classically, vanishing-off subgroups have been tremendously useful in determining links between the structure of $G$ and arithmetic properties of the degrees of its irreducible characters (see \cite{MI76} for details). To see the connection to Camina triples, we note that whenever $(G,N,M)$ is a Camina triple, every member of $\mathrm{Irr}(N\mid M)$ induces homogeneously to $G$, which means that every member of $\mathrm{Irr}(G\mid M)$ vanishes off of $N$. In particular, if $(G,N,M)$ is a Camina triple, the subgroup
\[\mathbf{V}(G\mid M)=\prod_{\chi\in\mathrm{Irr}(G\mid M)}\mathbf{V}(\chi)\] must contained in $N$. 

Of course, there is nothing about the definition of $\mathbf{V}(\chi)$ that requiries $\chi$ to be irreducible. We thus define for a supercharacter theory $\mathsf{S}$ of $G$ and an $\mathsf{S}$-normal subgroup $N$ of $G$ the subgroup $\mathbf{V}(\mathsf{S}\mid N)$ to be the product of the $\mathbf{V}(\chi)$ where $\chi$ ranges over all supercharacters of $\mathsf{S}$ that do not contain $N$ in their kernel. Our first main result is a generalization of \cite[Theorem 2.1]{NM14}.

\begin{introthm}\label{introdelta}
Let $\mathsf{S}$ be a supercharacter theory of $G$, and let $M\le N$ be $\mathsf{S}$-normal. The following are equivalent:
\begin{enumerate}[label={\bf(\arabic*)}]\openup5pt
\item$\mathsf{S}$ is a $\Delta$-product over $M$ and $N$;
\item for each $g\in G\setminus N$ and for each supercharacter $\chi$ not containing $M$ in its kernel, $\chi(g)=0$;
\item ${\bf V}(\mathsf{S}\mid M)\le N$.
\end{enumerate}
\end{introthm}

Notice that that Theorem~\ref{introdelta} gives a character-theoretic characterization of $\Delta$-products. Since all $\mathsf{S}$-normal subgroups of a supercharacter theory can be determined from the values of the characters, it follows that one may determine if $\mathsf{S}$ is a nontrivial $\Delta$-product with the same information. In particular, this can be determined from the supercharacter table of $\mathsf{S}$---an analog of the character table where the rows are indexed by supercharacters and the columns are indexed by superclasses. Just like the character table of $G$, the supercharacter table of $\mathsf{S}$ contains a lot of information about $\mathsf{S}$ in a condensed form. For example, one may find all of $\mathsf{S}$-normal subgroups of $G$, all of the $\mathsf{S}$-superclass sizes, as well as analogs $\mathsf{Z}(\mathsf{S})$ of the center of $G$, and $[G,\mathsf{S}]$ of the commutator subgroup of $G$. Other properties beyond the scope this paper can also be found (e.g. see \cite{SB18nil}). As a consequence of Theorem~\ref{introdelta}, we have that one may also determine if $\mathsf{S}$ is a nontrivial $\Delta$-product from the supercharacter table of $\mathsf{S}$. This gives an affirmative answer a question posed to the first author from his Ph.D. advisor Dr. Nathaniel Thiem.

Specializing the case where $M=N$, we have a similar result for $\ast$-products. 

\begin{introthm}\label{introstarprod}Let $\mathsf{S}$ be a supercharacter theory of $G$, and let $N$ be $\mathsf{S}$-normal. The following are equivalent:
\begin{enumerate}[label={\bf(\arabic*)}]\openup5pt
\item$\mathsf{S}$ is a $\ast$-product over $N$;
\item for each $g\in G\setminus N$ and for each supercharacter $\chi$ not containing $N$ in its kernel, $\chi(g)=0$; and, for each $g\in N$, there exists a supercharacter $\chi$ satisfying $N\nsubseteq\ker(\chi)$ and $\chi(g)\neq0$;
\item${\bf V}(\mathsf{S}\mid N)=N$.
\end{enumerate}
\end{introthm}

Some of this work is part of first author’s Ph.D. thesis at the University of Colorado Boulder under the supervision of Nathaniel Thiem. The first author expresses his gratitude to Dr. Thiem for the advice and support accompanying this work.

\section{Supercharacter theory}
In this section, we review some basics of supercharacter theory and supernormality, as well as set some notation. We then discuss two subgroups that can be defined for any supercharacter theory of a finite group $G$. 

Let $\mathrm{cf}(G)$ be the space of class functions of $G$; i.e.
\[\mathrm{cf}(G)=\{\alpha:G\to\mathbb{C}\mid \alpha(x^{-1}gx)=\alpha(g)\ \text{for all $x,g\in G$}\}.\]
Then $\mathrm{cf}(G)$ is a unital algebra under the convolution product $\ast$, defined by
\[(\alpha\ast\beta)(g)=\frac{1}{\norm{G}}\sum_{x\in G}\alpha(xg)\beta(x^{-1}).\]
It is easy to see that $\mathrm{cf}(G)$ is also a unital algebra under the pointwise product. For our purposes, a {\bf supercharacter theory} $\mathsf{S}$ of a finite group $G$ a subspace of $\mathrm{cf}(G)$ that is a unital subalgebra with respect to both the convolution and pointwise products. This subalgebra is semisimple with respect to both product, and by finding bases of idemptotents one obtains a distinguished set of characters $\mathrm{Irr}(\mathsf{S})$, called $\mathsf{S}$-{\bf irreducible characters} and a partition $\mathrm{Cl}(\mathsf{S})$ of $G$ into $\mathsf{S}$-{\bf classes}. Moreover, the sets $\mathrm{Irr}(\chi)$ of irreducible constituents of $\chi$, where $\chi$ ranges over the elements of $\mathrm{Irr}(\mathsf{S})$ gives a partition of $\mathrm{Irr}(G)$, and 
\[\chi=\sum_{\xi\in\mathrm{Irr}(\chi)}\xi(1)\xi\]
for each $\chi\in\mathrm{Irr}(\mathsf{S})$.

We remark here that our definition of $\mathrm{Irr}(\mathsf{S})$ yields the equation
\[\norm{G}=\sum_{\chi\in\mathrm{Irr}(\mathsf{S})}\chi(1),\]
which appears strange when compared to familiar equation
\[\norm{G}=\sum_{\chi\in\mathrm{Irr}(G)}\chi(1)^2.\]
Of course, from the definition of the elements of $\mathrm{Irr}(\mathsf{S})$, the first statement follows from the second. Also, in the event that $\mathsf{S}=\mathsf{m}(G)$, the usual character theory of $G$, the elements of $\mathrm{Irr}(\mathsf{S})$ have the form $\chi(1)\chi$ for some $\chi\in\mathrm{Irr}(G)$, which explains the perceived discrepancy. 
\begin{rem}
To see the original definition of supercharacter theory given by Diaconis and Isaacs, see \cite{ID07}. To see that the original definition and the one given here are equivalent, one may see, for example, \cite{SBJH18} or \cite{AH12}.
\end{rem}

Given a subset $X\subseteq\mathrm{Irr}(G)$, we will write $\sigma_X=\sum_{\xi\in X}\xi(1)\xi$. In particular, this means that $\chi=\sigma_{\mathrm{Irr}(\chi)}$ for each $\chi\in\mathrm{Irr}(\mathsf{S})$. For each $g\in G$, we will write $\mathrm{cl}_{\mathsf{S}}(g)$ for the $\mathsf{S}$-class containing $G$. We will also let $\mathrm{SCT}(G)$ denote the set of all supercharacter theories of $G$.

A group $G$ will always be assumed finite. Let $\mathsf{S}$ be a supercharacter theory of $G$. A subgroup $N$ of $G$ is called $\mathsf{S}$-{\bf normal} if $N$ is a union of $\mathsf{S}$-classes of $G$. In this event, we will write $N\lhd_{\mathsf{S}}G$. Given an $\mathsf{S}$-normal subgroup $N$ of $G$, Hendrickson showed in \cite[Propostion 6.4]{AH12} that one may associate a supercharacter theory $\mathsf{S}_N$ of $N$, where
\[\mathrm{Irr}(\mathsf{S}_N)=\bigl\{\sigma_{\mathrm{Irr}(\chi_N)}:\chi\in\mathrm{Irr}(\mathsf{S})\}\]
and
\[\mathrm{Cl}(\mathsf{S}_N)=\{\mathrm{cl}_{\mathsf{S}}(g):g \in N\}.\]
It was also shown that the quotient $G/N$ has an induced supercharacter theory $\mathsf{S}_{G/N}$ as well --- the $\mathsf{S}_{G/N}$-irreducible characters are just the $\mathsf{S}$-irreducible characters that contain $N$ in their kernels, naturally considered characters of $G/N$, and the $\mathsf{S}_{G/N}$-classes are the images of the $\mathsf{S}$-classes under the canonical surjection $G\to G/N$. 

As an immediate consequence of the construction of $\mathsf{S}_{G/N}$, we have that 
\[\mathrm{Irr}(G/N)=\bigcup_{\substack{\chi\in\mathrm{Irr}(\mathsf{S})\\N\subseteq\ker(\chi)}}\mathrm{Irr}(\chi),\]
where the characters of $G/N$ are being naturally identified with the characters of $G$ with $N$ in their kernel. In particular, this means that every $\mathsf{S}$-normal subgroup arises as the intersection of kernels of some $\mathsf{S}$-irreducible characters, a fact proved by Marberg (see \cite[Proposition 2.1]{EM11}). 

Given these facts, it will be beneficial later to introduce some notation. We will denote by $\mathrm{Irr}(\mathsf{S}/N)$ the set of $\mathsf{S}$-irreducible characters with $N$ in their kernels, and we will let $\mathrm{Irr}(\mathsf{S}\mid N)$ denote the set of  $\mathsf{S}$-irreducible characters not containing $N$ in their kernels. In particular, $\mathrm{Irr}(\mathsf{S})$ is the disjoint union of $\mathrm{Irr}(\mathsf{S}/N)$ and $\mathrm{Irr}(\mathsf{S}\mid N)$.  

The induced theories described above allow for one to take supercharacter theories of normal subgroups and quotients of $G$, and \enquote{glue} them together to form supercharacter theories of $G$, as long as some suitable conditions are satisfied. This construction is known as a $\Delta$-product.
\begin{thm}[{\normalfont\cite[Theorem 7.2]{AH12}}]\label{deltaproddef}
Let $H$ and $N$ be normal subgroups of $G$, with $H\le N$. Let $\mathsf{S}$ be a supercharacter theory of $N$ for which $H\lhd_{\mathsf{S}}N$, and let $\mathsf{T}$ be a supercharacter theory of $G/H$ for which $N/H\lhd_{\mathsf{T}}G/H$. Assume furthermore that every $\mathsf{S}$-class is a union of $G$-conjugacy classes and that $\mathsf{S}_{N/H}=\mathsf{T}_{N/H}$. Then $\mathsf{E}$ is a supercharacter theory of $G$, where
\[\mathrm{Irr}(\mathsf{E})=\bigl\{\chi^G:\chi\in\mathrm{Irr}(\mathsf{S}\mid H)\bigr\}\cup\mathrm{Irr}(\mathsf{T}),\]
and
\[\mathrm{Cl}(\mathsf{E})=\mathrm{Cl}(\mathrm{S})\cup\{\pi^{-1}(K):K\in\mathsf{T},\ N/H\nsubseteq K\},\]
and $\pi:G\to G/H$ is the canonical projection.
\end{thm}
The supercharacter theory of Theorem~\ref{deltaproddef} is called the $\Delta$-product of $\mathsf{S}$ and $\mathsf{T}$, and is denoted $\mathsf{S}\mathop{\Delta}\mathsf{T}$. In the special case that $H=N$, the $\Delta$-product is called a $\ast$-product, and $\mathsf{S}\mathop{\Delta}\mathsf{T}$ is denoted by $\mathsf{S}\ast\mathsf{T}$. That is, 
\[\mathrm{Irr}(\mathsf{S}\ast\mathsf{T})=\bigl\{\chi^G:\chi\in\mathrm{Irr}(\mathsf{S})\setminus\{\mathbbm{1}\}\bigr\}\cup\mathrm{Irr}(\mathsf{T}),\]
and
\[\mathrm{Cl}(\mathsf{S}\ast\mathsf{T})=\mathrm{Cl}(\mathrm{S})\cup\{\pi^{-1}(K):K\in\mathsf{T}\setminus\{H\}\},\]
and $\pi:G\to G/H$ is the canonical projection.

\section{The subgroups ${\bf Z}(\mathsf{S})$ and $[G,\mathsf{S}]$}
In this section, we discuss $\mathsf{S}$-normal generalizations of the center and commutator subgroup of $G$. Let $\mathsf{S}$ be a supercharacter theory of $G$, and let
\[{\bf Z}(\mathsf{S})=\{g\in G:\norm{\mathrm{cl}_{\mathsf{S}}(g)}=1\}.\]
It is well-known that ${\bf Z}(\mathsf{S})$ is a subgroup of $G$, and therefore it is an $\mathsf{S}$-normal subgroup of $G$ contained in the center of $G$. 

The group $\mathbf{Z}(\mathsf{S})$ has a lot in common with its classical counterpart, as is illustrated by the following basic results, all appearing in \cite{SB18nil}, will be needed later.

\begin{lem}[{\normalfont\cite[Lemma 3.4]{SB18nil}}]\label{centerrestriction}
Let $\mathsf{S}$ be a supercharacter theory of $G$, and let $\chi\in\mathrm{Irr}(\mathsf{S})$. Then $\chi_{{\bf Z}(\mathsf{S})}=\chi(1)\vartheta_\chi$ for some irreducible character $\vartheta_\chi$ of ${\bf Z}(\mathsf{S})$.
\end{lem}
\begin{lem}[{\normalfont\cite[Proposition 3.5 (3)]{SB18nil}}]\label{centercharacter}
Let $\mathsf{S}$ be a supercharacter theory of $G$. Then
\[\mathbf{Z}(\mathsf{S})=\bigcap_{\chi\in\mathrm{Irr}(\mathsf{S})}\mathbf{Z}(\chi).\]
\end{lem}

\begin{lem}[{\normalfont\cite[Proposition 3.5 (4)]{SB18nil}}]\label{centercharacter}
Let $\mathsf{S}$ be a supercharacter theory of $G$, let $\chi\in\mathrm{Irr}(\mathsf{S})$ and write $Z={\bf Z}(\chi)$. Then \[Z/\ker(\chi)={\bf Z}(\mathsf{S}_{G/\ker(\chi)}),\] and is cyclic.
\end{lem}

%%%%%%%%%%%%%%%%%%%%%%%%%%%%%%%%%%%%%%%%%%%%%%%%%%%%%%%%%%%%%%%%%

We now describe an analog of the commutator subgroup, as introduced in \cite{SB18nil}. Let $H\le G$. Define the subgroup $[H,\mathsf{S}]$ by
\[[H,\mathsf{S}]=\inner{g^{-1}k:g\in H,\ k\in\mathrm{cl}_{\mathsf{S}}(g)}.\]
The following results generalize some basic properties of commutator subgroups, and will be needed in the remainder of this paper. The first one illustrates why this generalization of the commutator subgroup is useful from the lens of supercharacter theory.

\begin{lem}[{\normalfont\cite[Corollary 3.8]{SB18nil}}]\label{snormalcommutator}
Let $\mathsf{S}$ be a supercharacter theory of $G$, and let $N$ be an $\mathsf{S}$-normal. Then $[N,\mathsf{S}]=\inner{g^{-1}k:g\in N,\ k\in\mathrm{cl}_{\mathsf{S}}(g)}$ is $\mathsf{S}$-normal.
\end{lem}

The next result describes exactly which $\mathsf{S}$-characters are linear. 
\begin{lem}[{\normalfont\cite[Proposition 3.11]{SB18nil}}]\label{gsproperty}
Let $\mathsf{S}$ be a supercharacter theory of $G$, and let $N$ be $\mathsf{S}$-normal. Then $[G,\mathsf{S}]\le N$ if and only if ${\bf Z}(\mathsf{S}_{G/N})=G/N$.
\end{lem}

\section{Vanishing-off subgroup}

Let $\chi$ be a character of $G$. The {\bf vanishing-off subgroup} of $\chi$ is the subgroup 
\[{\bf V}(\chi)=\inner{g\in G:\chi(g)\neq0}.\]
Then ${\bf V}(\chi)$ is the smallest subgroup, $V\le G$ such that $\chi$ vanishes on $G\setminus V$. The vanishing-off subgroup is used extensively in Chapter 12 of \cite{MI76} to prove theorems relating the degrees of the irreducible characters $G$ to the structure of $G$. The following is a generalization of a well-known observation along these lines. We use $\inner{\ ,\ }$ to denote the usual inner product of characters (e.g. see \cite[Definition 2.16]{MI76}); that is for class functions $\alpha$ and $\beta$
\[\inner{\alpha,\beta}=\frac{1}{\norm{G}}\sum_{g\in G}\alpha(g)\overline{\beta(g)}.\]

\begin{lem}\label{indexsize}
Let $\mathsf{S}$ be a supercharacter theory of $G$, and let $N$ be $\mathsf{S}$-normal. Let $\chi\in\mathrm{Irr}(\mathsf{S}\mid N)$, and let $\psi\in\mathrm{Irr}(\mathsf{S}_N)$ satisfy $\inner{\psi,\chi_N}>0$. Then
\[\chi(1)/\psi(1)\le\norm{G:N},\]
with equality if and only if $\mathbf{V}(\chi)\le N$.
\end{lem}
\begin{proof}
Write $\alpha=\chi(1)/\psi(1)$. Then, we have 
\begin{align*}\alpha^2\psi(1)&=\inner*{\chi_N,\chi_N}=\frac{1}{\norm{N}}\sum_{g\in N}\norm{\chi(g)}^2\\
&\le\frac{1}{\norm{N}}\sum_{g\in G}\norm{\chi(g)}^2=\norm{G:N}\chi(1),
\end{align*}
with equality if and only if $\chi$ vanishes on $G\setminus N$. The result follows.
\end{proof} 

Note that if $\chi$ is linear, then ${\bf V}(\chi)=G$. For this reason, Lewis \cite{MLvos09} defines the vanishing-off subgroup of $G$ to be the subgroup
\[{\bf V}(G)=\inner{g\in G: \text{there exists}\ \, \chi\in\mathrm{Irr}(G\mid[G,G])\ \,\text{such that}\ \,\chi(g)\neq0}.\]
In \cite{MLvos09}, Lewis uses ${\bf V}(G)$ to prove several results about Camina pairs ---$(G,N)$ is a called a Camina pair if every class lying outside of $N$ is a union of $N$-cosets--- and {\it generalized} Camina pairs (Camina triples of the form $(G,N,[G,G])$). Extending these results, Mlaiki proves analogous results about Camina triples in \cite{NM14}, where he introduces the notation 
\[{\bf V}(G\mid N)=\inner{g\in G: \text{there exists}\ \, \chi\in\mathrm{Irr}(G\mid N)\ \,\text{such that}\ \,\chi(g)\neq0}.\]
With our analogs $[G,\mathsf{S}]$ and $\mathrm{Irr}(\mathsf{S}\mid N)$, we define supercharacter theoretic versions of these subgroups as follows. Let $\mathsf{S}$ be a supercharacter theory of $G$, and let $N$ be an $\mathsf{S}$-normal subgroup of $G$. Define \label{svanishN}
\[{\bf V}(\mathsf{S}\mid N)=\inner{g\in G:\text{there exists}\ \, \chi\in\mathrm{Irr}(\mathsf{S}\mid N)\ \,\text{such that}\ \,\chi(g)\neq0}\] and the $\mathsf{S}$-{\bf vanishing-off subgroup}, \label{svanish} ${\bf V}(\mathsf{S})$, by
\[{\bf V}(\mathsf{S})={\bf V}(\mathsf{S}\mid[G,\mathsf{S}]).\]
Note that when $\mathsf{S}=\mathsf{m}(G)$, ${\bf V}(\mathsf{S})={\bf V}(G)$, the vanishing-off subgroup as defined in \cite{MLvos09}. Lemma 3.2(7) of that paper shows that $\mathbf{V}(G)$ is the product of the vanishing-off subgroups of all nonlinear characters. Essentially the same proof verifies the definition of $\mathbf{V}(G\mid N)$ given in the introduction. 
In fact, the same idea generalizes the result for the supercharacter theoretic versions.

\begin{prop}\label{vanishingprod}Let $\mathsf{S}$ be a supercharacter theory of $G$, and let $N$ be $\mathsf{S}$-normal. Then ${\bf V}(\mathsf{S}\mid N)\lhd_{\mathsf{S}}G$ and 
\[{\bf V}(\mathsf{S}\mid N)=\prod_{\chi\in\mathrm{Irr}(\mathsf{S}\mid N)}{\bf V}(\chi),\]
where ${\bf V}(\chi)=\inner{g\in G:\chi(g)\neq0}$.
\end{prop}
\begin{proof}
If $g\in G$ such that $\psi(g)\neq0$ for some $\psi\in\mathrm{Irr}(\mathsf{S}\mid N)$, then $g\in{\bf V}(\psi)\subseteq\prod_\chi{\bf V}(\chi)$. It follows that
\[\mathbf{V}(\mathsf{S}\mid N)\subseteq\prod_{\chi\in\mathrm{Irr}(\mathsf{S}\mid N)}{\bf V}(\chi).\]
As the reverse containment is clear, the second statement is proved. 

Since ${\bf V}(\chi)$ is generated by a union of $\mathsf{S}$-classes, it follows that ${\bf V}(\chi)\lhd_{\mathsf{S}}G$ for each $\chi\in\mathrm{Irr}(\mathsf{S}\mid N)$. Since ${\bf V}(\mathsf{S}\mid N)$ is the product of $\mathsf{S}$-normal subgroups, ${\bf V}(\mathsf{S}\mid N)$ is also $\mathsf{S}$-normal.
\end{proof}

The following results show that the function taking an $\mathsf{S}$-normal subgroup $N$ to $\mathbf{V}(\mathsf{S}\mid N)$ satisfies some desirable properties. In particular, it is order-preserving and respects products.
\begin{lem}\label{mono}
Let $\mathsf{S}$ be a supercharacter theory of $G$, and let $H$ and $N$ be $\mathsf{S}$-normal subgroups. If $H\le N$, then ${\bf V}(\mathsf{S}\mid H)\le{\bf V}(\mathsf{S}\mid N)$.
\end{lem}
\begin{proof}
This follows from Proposition~\ref{vanishingprod} and the observation that $\mathrm{Irr}(\mathsf{S}\mid H)\subseteq\mathrm{Irr}(\mathsf{S}\mid N)$ whenever $H\le N$.
\end{proof}

\begin{lem}\label{vanprod}
Let $\mathsf{S}$ be a supercharacter theory of $G$, and let $N_1$ and $N_2$ be $\mathsf{S}$-normal. Then ${\bf V}(\mathsf{S}\mid N_1N_2)={\bf V}(\mathsf{S}\mid N_1)\mspace{2mu}{\bf V}(\mathsf{S}\mid N_2)$.
\end{lem}
\begin{proof}
First note that $\mathrm{Irr}(\mathsf{S}\mid N_i)\subset\mathrm{Irr}(\mathsf{S}\mid N_1N_2)$ for $i=1,2$. If $\chi\notin\mathrm{Irr}(\mathsf{S}\mid N_1)\cup\mathrm{Irr}(\mathsf{S}\mid N_2)$, then $N_1, N_2\le\ker(\chi)$, which means that $\chi\notin\mathrm{Irr}(\mathsf{S}\mid N_1N_2)$. Hence 
\[\mathrm{Irr}(\mathsf{S}\mid N_1N_2)= \mathrm{Irr}(\mathsf{S}\mid N_1)\cup\mathrm{Irr}(\mathsf{S}\mid N_2).\]
By Proposition~\ref{vanishingprod}, we have
\begin{align*}{\bf V}(\mathsf{S}\mid N_1N_2)&=\prod_{\chi\in\mathrm{Irr}(\mathsf{S}\mid N_1N_2)}{\bf V}(\chi)=\prod_{\chi\in\mathrm{Irr}(\mathsf{S}\mid N_1)\cup\mathrm{Irr}(\mathsf{S}\mid N_2)}{\bf V}(\chi)\\[3ex]
&=\left(\prod_{\chi\in\mathrm{Irr}(\mathsf{S}\mid N_1)}{\bf V}(\chi)\right)\left(\prod_{\chi\in\mathrm{Irr}(\mathsf{S}\mid N_2)}{\bf V}(\chi)\right)={\bf V}(\mathsf{S}\mid N_1)\mspace{2mu}{\bf V}(\mathsf{S}\mid N_2).
\end{align*}
\end{proof}

The vanishing off subgroup also behaves predictably with quotients.
\begin{lem}Let $\mathsf{S}$ be a supercharacter theory of $G$, let $N$ and $H$ be $\mathsf{S}$-normal, and assume that $N\le H$. Then
\[{\bf V}(\mathsf{S}_{G/N}\mid H/N)\,{\bf V}(\mathsf{S}\mid N)/N={\bf V}(\mathsf{S}\mid H)/N.\]
\end{lem}
\begin{proof}
Since $N\le H$, we have $\mathrm{Irr}(\mathsf{S}\mid N)\subseteq\mathrm{Irr}(\mathsf{S}\mid H)$. In fact, 
\[\mathrm{Irr}(\mathsf{S}\mid H)=\mathrm{Irr}(\mathsf{S}\mid N)\cup\bigl(\mathrm{Irr}(\mathsf{S}\mid H)\cap\mathrm{Irr}(\mathsf{S}/N)\bigr).\]
Therefore
\[
{\bf V}(\mathsf{S}\mid H)={\bf V}(\mathsf{S}\mid N)\cdot\prod_{\substack{\chi\in\mathrm{Irr}(\mathsf{S}/N)\\H\nsubseteq\ker(\chi)}}{\bf V}(\chi).\]
Note that $N\le\mathbf{V}(\mathsf{S}\mid H)$ since $N\le H$. Also note that if $N\le\ker(\chi)$, and $\tilde{\chi}$ denotes the deflated character of $G/N$, then $\mathbf{V}(\chi)/N=\mathbf{V}(\tilde{\chi})$. Therefore the result follows by reducing modulo $N$.
\end{proof}

%\begin{prop}
%Let $\mathsf{S}$ be a supercharacter theory of $G$, and let $M$ be an $\mathsf{S}$-normal subgroup of $G$. If ${\bf V}(\mathsf{S})$ is a proper subgroup of $M$, then $[M,\mathsf{S}]=[G,\mathsf{S}]$.
%\end{prop}
%\begin{proof}
%Let $m\in M\setminus{\bf V}(\mathsf{S})$. Then every nonlinear $\mathsf{S}$-character vanishes at $m$. From Lemma~\ref{column}, this means that $\norm{\mathrm{cl}_{\mathsf{S}}(m)}=\norm{[G,\mathsf{S}]}$. Since $\mathrm{cl}_{\mathsf{S}}(m)\subseteq m[G,\mathsf{S}]$, this forces $\mathrm{cl}_{\mathsf{S}}(m)=m[G,\mathsf{S}]$. In particular, for every $y\in[G,\mathsf{S}]$, there exists $k\in\mathrm{cl}_{\mathsf{S}}(m)$ such that $y=m^{-1}k$. Since $m^{-1}k\in [M,\mathsf{S}]$, we have $[G,\mathsf{S}]\subseteq[M,\mathsf{S}]$. Since the reverse inclusion is clear, the proof is complete.
%\end{proof}

\section{$\Delta$-products}
Let $\mathsf{S}$ be a supercharacter theory of $G$, and assume that $M$ and $N$ are $\mathsf{S}$-normal subgroups of $G$ satisfying $M\le N$. Then $\mathsf{S}$ is a nontrivial $\Delta$-product over $M$ and $N$ if and only if every $\mathsf{S}$-class lying outside of $N$ is a union of $M$-cosets \cite[Proposition 7.3]{AH12}. In this section, we determine several other equivalent characterizations of $\Delta$-products. In particular, we prove Theorem~\ref{introdelta} with some other characterizations. Before doing so, however, we will need the following generalization of column orthogonality, as well as a useful consequence.

\begin{lem}[{\normalfont\cite[Theorem 3.3]{SB18nil}}]\label{column}Let $\mathsf{S}$ be a supercharacter theory of $G$. For each $g,h\in G$, we have
\[{}\sum_{\chi\in\mathrm{Irr}(\mathsf{S})}\frac{\chi(g)\overline{\chi(h)}}{\chi(1)}=\frac{\norm{G}}{\norm{\mathrm{cl}_{\mathsf{S}}(g)}}\]
if $h\in\mathrm{cl}_{\mathsf{S}}(g)$, and is 0 otherwise.
\end{lem}

\begin{lem}\label{vanishingcontainment}Let $\mathsf{S}$ be a supercharacter theory of $G$, and let $N$ be $\mathsf{S}$-normal. Then $N$ is contained in ${\bf V}(\mathsf{S}\mid N)$.
\end{lem}
\begin{proof}
Suppose that $N$ is not contained in ${\bf V}(\mathsf{S}\mid N)$, and let $g\in N\setminus {\bf V}(\mathsf{S}\mid N)$. Then $\chi(g)=0$ for all $\chi\in\mathrm{Irr}(\mathsf{S}\mid N)$ and is in the kernel of every other $\mathsf{S}$-irreducible character. Therefore, Lemma~\ref{column} gives
\begin{align*}\frac{\norm{G}}{\norm{\mathrm{cl}_{\mathsf{S}}(g)}}&={}\sum_{\chi\in\mathrm{Irr}(\mathsf{S})}\frac{\norm{\chi(g)}^2}{\chi(1)}
={}\sum_{\chi\in\mathrm{Irr}(\mathsf{S}/N)}\frac{\norm{\chi(g)}^2}{\chi(1)}+{}\sum_{\chi\in\mathrm{Irr}(\mathsf{S}\mid N)}\frac{\norm{\chi(g)}^2}{\chi(1)}\\
&={}\sum_{\chi\in\mathrm{Irr}(\mathsf{S}/N)}\frac{\chi(1)^2}{\chi(1)}=\sum_{\chi\in\mathrm{Irr}(\mathsf{S}/N)}\chi(1)=\sum_{\xi\in\mathrm{Irr}(G/N)}\xi(1)^2=\frac{\norm{G}}{\norm{N}}.
\end{align*}
This implies that $\mathrm{cl}_{\mathsf{S}}(g)=N$, which is a contradiction since $g\in N$. The result follows.
\end{proof}

We now prove Theorem~\ref{introdelta}.

\begin{thm}\label{delta}
Let $\mathsf{S}$ be a supercharacter theory of $G$, and let $M\le N$ be $\mathsf{S}$-normal subgroups. The following are equivalent:
\begin{enumerate}[label={\bf(\arabic*)}]\openup5pt
\item$\mathsf{S}$ is a $\Delta$-product over $M$ and $N$;
\item for each $g\in G\setminus N$, $\norm{\mathrm{cl}_{\mathsf{S}}(g)}=\norm{\mathrm{cl}_{\mathsf{S}_{G/M}}(gM)}\norm{M}$;
\item for each $g\in G\setminus N$, $\chi(g)=0$ for each $\chi\in\mathrm{Irr}(\mathsf{S}\mid M)$;
\item ${\bf V}(\mathsf{S}\mid M)\le N$;
\item for each $g\in G\setminus N$, and $m\in M$, there exists $k\in \mathrm{cl}_{\mathsf{S}}(g)$ such that $g^{-1}k=m$.
\end{enumerate}
\end{thm}
\begin{proof}
First, we show that {\bf(1)} implies {\bf(2)}. If $\mathsf{S}$ is a $\Delta$-product over $M$ and $N$, then every class lying outside of $N$ is a union of $M$-cosets by \cite[Proposition 4.3]{AH08}, which implies {\bf(2)}. 

Next, we show {\bf(2)} implies {\bf(3)}. By Lemma~\ref{column}, we have
\begin{align*}\frac{\norm{G}}{\norm{\mathrm{cl}_{\mathsf{S}}(g)}}&={}\sum_{\chi\in\mathrm{Irr}(\mathsf{S})}\frac{\norm{\chi(g)}^2}{\chi(1)}
={}\sum_{\chi\in\mathrm{Irr}(\mathsf{S}/M)}\frac{\norm{\chi(g)}^2}{\chi(1)}+{}\sum_{\chi\in\mathrm{Irr}(\mathsf{S}\mid M)}\frac{\norm{\chi(g)}^2}{\chi(1)}\\
&={}\sum_{\chi\in\mathrm{Irr}(\mathsf{S}/M)}\frac{\norm{\tilde{\chi}(gM)}^2}{\chi(1)}+{}\sum_{\chi\in\mathrm{Irr}(\mathsf{S}\mid M)}\frac{\norm{\chi(g)}^2}{\chi(1)}\\
&=\frac{\norm{G/M}}{\norm{\mathrm{cl}_{\mathsf{S}_{G/M}}(gM)}}+{}\sum_{\chi\in\mathrm{Irr}(\mathsf{S}\mid M)}\frac{\norm{\chi(g)}^2}{\chi(1)}=\frac{\norm{G}}{\norm{\mathrm{cl}_{\mathsf{S}}(g)}}+{}\sum_{\chi\in\mathrm{Irr}(\mathsf{S}\mid M)}\frac{\norm{\chi(g)}^2}{\chi(1)},
\end{align*}
which means that 
\[{}\sum_{\chi\in\mathrm{Irr}(\mathsf{S}\mid M)}\frac{\norm{\chi(g)}^2}{\chi(1)}=0.\]
Since each term in this sum is nonnegative, this means that $\chi(g)=0$ for each $\chi\in\mathrm{Irr}(\mathsf{S}\mid M)$.

The fact that {\bf(3)} is equivalent to {\bf(4)}  follows from the definition of ${\bf V}(\mathsf{S}\mid M)$, and that {\bf(5)} and {\bf(1)} are equivalent follows from \cite[Proposition 4.3]{AH08}. So we need only show that {\bf(3)} implies {\bf(5)} to complete the proof. Assume {\bf(3)}. Then the supercharacters of $\mathsf{S}$ lying in $\mathrm{Irr}(\mathsf{S}\mid M)$ are constant on $G\setminus N$ and so on the cosets of $M$ lying outside of $N$. It is clear that the $\mathsf{S}$-irreducible characters with $M$ in their kernels are also constant on the cosets of $M$ lying outside of $N$, so every supercharacter of $\mathsf{S}$ is constant on each coset of $M$ lying outside of $N$. Hence each superclass lying outside of $N$ is a union of $M$-cosets, which gives {\bf(5)}.
\end{proof}

Recall that every $\mathsf{S}$-normal subgroup can be found from the $\mathsf{S}$-character table. Therefore, as a corollary to Theorem~\ref{delta}, it can be deduced from the $\mathsf{S}$-character table whether or not $\mathsf{S}$ is a nontrivial $\Delta$-product.
\begin{cor}\label{deltatable}If $\mathsf{S}$ is $\Delta$-product over $M$ and $N$, this can be deduced from the $\mathsf{S}$-character table using {\bf(3)} of Theorem~\ref{delta}.
\end{cor} 
\begin{proof}
A supercharacter theory $\mathsf{S}$ is a $\Delta$-product over $\mathsf{S}$-normal subgroups $M$ and $N$ if and only if condition {\bf(3)} of Theorem~\ref{delta} holds. Since every $\mathsf{S}$-normal subgroup $H$ of $G$ is the intersection of the kernels of some $\mathsf{S}$-irreducible characters, specifically those that contain $H$ in their kernel, all of this information can be read off the $\mathsf{S}$-character table.
\end{proof}
\begin{exam}\label{deltaexample}
Let $G=C_{30}$, and let $c_1=(\sqrt{5}-1)/2$ and $c_2=-(\sqrt{5}+1)/2$. The table 
\begin{center}
\begin{tabular}{|c|c|c|c|c|c|}
\hline
$\mathsf{S}$&1&$K_1$&$K_2$&$K_3$&$K_4$\\
\hline
$\mathbbm{1}$&1&1&1&1&1\\
\hline
$\chi_1$&1&$-1$&1&1&1\\
\hline
$\chi_2$&4&0&$-2$&4&4\\
\hline
$\chi_3$&12&0&0&$6c_1$&$6c_2$\\
\hline
$\chi_4$&12&0&0&$6c_2$&$6c_1$\\
\hline
\end{tabular}\\[5ex]
\end{center}

\noindent is the $\mathsf{S}$-character table for a supercharacter theory $\mathsf{S}$ of $G$. From the table, we see that the set of $\mathsf{S}$-normal subgroups is $\{1,H,N\}$, where $H=\{1\}\cup K_3\cup K_4$ and $N=K_2\cup H$. Using Lemma~\ref{column}, we find that $\norm{K_2}=10$, and $\norm{K_3}=\norm{K_4}=2$. It follows that $\norm{H}=5$ and $\norm{N}=15$. The $\mathsf{S}$-irreducible characters that do not have $N$ in their kernels are $\{\chi_2,\chi_3,\chi_4\}$ and each of these vanishes on $G\setminus N=K_1$. Therefore, this supercharacter theory is a nontrivial $\Delta$-product over $H$ and $N$. 
\end{exam}
We have the following corollaries to Theorem~\ref{delta}. The first gives some information about the singleton classes and linear characters of a $\Delta$-product.
\begin{cor}
Let $\mathsf{S}$ be a supercharacter theory of $G$. If $\mathsf{S}$ is a $\Delta$-product over $M$ and $N$, then ${\bf Z}(\mathsf{S})\le N$ and $M\le[G,\mathsf{S}]$.
\end{cor}
\begin{proof}
Since every $\mathsf{S}$-class lying outside of $N$ has at least $\norm{M}$ elements by {\bf(2)} of Theorem~\ref{delta}, we must have ${\bf Z}(\mathsf{S})\le N$. 

Since every $\mathsf{S}$-character lying in $\mathrm{Irr}(\mathsf{S}\mid M)$ vanishes outside of $N$ by {\bf(3)} of Theorem~\ref{delta}, we must that $\mathrm{Irr}(\mathsf{S}/[G,\mathsf{S}])\subseteq\mathrm{Irr}(\mathsf{S}/M)$. This means that $M\le [G,\mathsf{S}]$.
\end{proof}

\begin{cor}\label{quotientdelta}Let $\mathsf{S}$ be a supercharacter theory of $G$, and assume that $\mathsf{S}$ is a $\Delta$-product over $M$ and $N$. Let $H\lhd_{\mathsf{S}}G$ satisfy $HM<N$. Then $\mathsf{S}_{G/H}$ is a $\Delta$-product over $HM/H$ and $N/H$.\end{cor}
\begin{proof}
By Theorem~\ref{delta}, $\chi$ vanishes on $G\setminus N$ for all $\chi\in\mathrm{Irr}(\mathsf{S}\mid M)$. So every $\chi$ in $\mathrm{Irr}(\mathsf{S}/H)\cap\mathrm{Irr}(\mathsf{S}\mid M)$ vanishes outside of $N$. Since $\mathrm{Irr}(\mathsf{S}/H)\cap\mathrm{Irr}(\mathsf{S}\mid M)=\mathrm{Irr}(\mathsf{S}/H)\cap\mathrm{Irr}(\mathsf{S}\mid HM)$, this is equivalent to every $\chi$ in $\mathrm{Irr}(\mathsf{S}_{G/H}\mid HM/H)$ vanishing outside of $N/H$. The result follows from Theorem~\ref{delta}.
\end{proof}

The following result shows how one may \enquote{extend the gap} between the subgroups $M$ and $N$ for which $\mathsf{S}$ is a $\Delta$-product. 
\begin{cor}\label{pushapart}Let $\mathsf{S}$ be a supercharacter theory of $G$. Suppose that $\mathsf{S}$ is a $\Delta$-product over $M$ and $N$, and $K\lhd_{\mathsf{S}}G$. If $N\le K$, then $\mathsf{S}$ is a $\Delta$-product over $M$ and $K$. Also if $K\le M$, then $\mathsf{S}$ is a $\Delta$-product over $K$ and $N$.
\end{cor}
\begin{proof}
This is immediate from Lemma~\ref{mono} and the equivalent characterization {\bf(4)} of Theorem~\ref{delta}.
\end{proof}
As opposed to Corollary~\ref{pushapart}, the following two results shows how one may \enquote{shrink the gap.}
\begin{prop}Let $\mathsf{S}$ be a supercharacter theory of $G$, and let $M,N_1$, and $N_2$ be $\mathsf{S}$-normal subgroups of $G$ satisfying $M\subseteq N_1\cap N_2$. If $\mathsf{S}$ is a $\Delta$-product over $M$ and $N_1$, and also over $M$ and $N_2$, then $\mathsf{S}$ is a $\Delta$-product over $N_1\cap N_2$.
\end{prop}
\begin{proof}
Assume that $\mathsf{S}$ is a $\Delta$-product over $M$ and $N_1$, and also over $M$ and $N_2$. By Theorem~\ref{delta}, we have that ${\bf V}(\mathsf{S}\mid M)\le N_1\cap N_2$, which means that $\mathsf{S}$ is a $\Delta$-product over $M$ and $N_1\cap N_2$.
\end{proof}

\begin{prop}
Let $\mathsf{S}$ be a supercharacter theory of $G$ and let $M_1,M_2,$ and $N$ be $\mathsf{S}$-normal subgroups of $G$. If $\mathsf{S}$ is a $\Delta$-product over $M_1$ and $N$, and also over $M_2$ and $N$, then $\mathsf{S}$ is a $\Delta$-product over $M_1M_2$ and $N$.
\end{prop}
\begin{proof}
By Theorem~\ref{delta}, we have that ${\bf V}(\mathsf{S}\mid M_1)$ and ${\bf V}(\mathsf{S}\mid M_2)$ are contained in $N$. Hence so is their product, as is $M_1M_2$ by Lemma~\ref{vanishingcontainment}. By Lemma~\ref{vanprod}, we have
\[{\bf V}(\mathsf{S}\mid M_1M_2)={\bf V}(\mathsf{S}\mid M_1)\mspace{2mu}{\bf V}(\mathsf{S}\mid M_2)\le N,\]
so $\mathsf{S}$ is a $\Delta$-product over $M_1M_2$ and $N$ by Theorem~\ref{delta}.
\end{proof}

\section{$\ast$-products}
The $\ast$-product is a special example of a $\Delta$-product. Specifically, \enquote{$\mathsf{S}$ is an $\ast$-product over $N$} means \enquote{$\mathsf{S}$ is a $\Delta$-product over $N$ and $N$.} In this section, we tailor the results of the previous section to $\ast$-products.
%\begin{prop}\label{starprodization}Let $\mathsf{S}\in\mathrm{SCT}(G)$ and let $N\lhd_{\mathsf{S}}G$. The following are equivalent:
%\begin{enumerate}[label={\bf(\arabic*)}]\openup5pt
%\item$\mathsf{S}$ is a $\ast$-product over $N$;
%\item for each $g\in G\setminus N$ and $\chi\in\mathrm{Irr}(\mathsf{S}\mid N)$, $\chi(g)=0$; and, for each $g\in N$, there exists $\mathrm{Irr}(\mathsf{S}\mid N)$ for which $\chi(g)\neq0$;
%\item${\bf V}(\mathsf{S}\mid N)=N$.
%\end{enumerate}
%\end{prop}
\begin{proof}[Proof of Theorem~\ref{introstarprod}]
First we assume {\bf(1)} and prove {\bf(2)}. By Theorem~\ref{delta}, we have $\chi(g)=0$ for each $g\in G\setminus N$ and $\chi\in\mathrm{Irr}(\mathsf{S}\mid N)$. Now, let $g\in N$ and suppose that $\chi(g)=0$ for each $\chi\in\mathrm{Irr}(\mathsf{S}\mid N)$. By Theorem~\ref{column}, it follows that $\norm{\mathrm{cl}_{\mathsf{S}}(g)}=\norm{N}$. As this is impossible, no such $g$ exists and {\bf(2)} is proved. Now, {\bf(2)} implies {\bf(3)} is immediate and {\bf(3)} implies {\bf(1)} follows from Theorem~\ref{delta}. The proof is complete.
\end{proof}
\begin{rem}
It is clear from the definitions that $(G,N)$ is a Camina pair if and only if $\mathsf{m}(G)$ is a $\ast$-product over the Camina kernel $N$. As mentioned in the introduction, Lewis and Wynn proved in \cite{camina} that $\mathsf{m}(G)$ is a nontrivial $\ast$-product is enough to guarantee that every supercharacter theory of a group $G$ with a Camina kernel is a nontrivial $\ast$-product. This surprising result implies therefore that every supercharacter theory of a group $G$ is a nontrivial $\ast$-product if and only if $\mathsf{m}(G)$ is.
\end{rem}
Since a $\ast$-product is a special example of a $\Delta$-product, we already knew from Corollary~\ref{deltatable} that one may determine if $\mathsf{S}$ is a nontrivial $\ast$-product from the $\mathsf{S}$-character table. We also get this as a corollary to Theorem~\ref{introstarprod}.
\begin{cor}If $\mathsf{S}$ is $\ast$-product over $N$, this can be deduced from the $\mathsf{S}$-character table using {\bf(2)} of Theorem~\ref{introstarprod}.\end{cor} 
\begin{proof}
A supercharacter theory $\mathsf{S}$ is a $\ast$-product over $N\lhd_{\mathsf{S}}G$ if and only if condition {\bf(2)} of Theorem~\ref{introstarprod} holds. Since $N$ is the intersection of the kernels of all $\mathsf{S}$-irreducible characters that contain $N$ in their kernel, this information can be read off of the $\mathsf{S}$-character table.
\end{proof}

We mention the following consequence of Theorem~\ref{introstarprod}. Note that if $[G,G]\le N$, then every irreducible character of $G/N$ is linear. Therefore, if $\mathsf{S}$ is any supercharacter theory of $N$ whose superclasses of unions of $G$-classes, then the supercharacter theory $\mathsf{T}=\mathsf{S}\ast\mathsf{m}(G/N)$ of $G$ satisfies $N=\mathbf{V}(\mathsf{S})$. Theorem~\ref{introstarprod} shows that the converse of this statement also holds. 
\begin{cor}
Let $N$ be a normal subgroup of $G$. There is a supercharacter theory $\mathsf{S}$ of $G$ satisfying $\mathbf{V}(\mathsf{S})=N$ if and only if $G/N$ is abelian. In the event that $G/N$ is abelian and $N\lhd_{\mathsf{S}}G$, we have $\mathbf{V}(\mathsf{S})=N$ if and only if $\mathsf{S}=\mathsf{S}_N\ast\mathsf{m}(G/N)$.
\end{cor}
Note that this also shows (perhaps not surprisingly) that there is no obvious relationship between $\mathbf{V}(\mathsf{C})$ and $\mathbf{V}(\mathsf{D})$, even if $\mathsf{C}\preccurlyeq\mathsf{D}$. In particular, even though the condition $\mathbf{V}(G)<G$ is rather strong (in fact this condition guarantees that $G$ is essentially a $p$-group \cite[Lemma 2.6]{MLvos09}), the only thing that can be deduced from $G$ having a supercharacter $\mathsf{S}$ such that $\mathbf{V}(\mathsf{S})<G$ is that $G$ is not perfect (i.e. $[G,G]<G$).

\begin{exam}
We return to Example~\ref{deltaexample}, to further decompose $\mathsf{S}$. Let $\psi_i\in\mathrm{Irr}(\mathsf{S}_N)$ such that $\inner{\psi_i,(\chi_i)_N}>0$. Then by Lemma~\ref{indexsize}, we have $\chi_i(1)/\psi_i(1)\in\{1,2\}$. Thus, by order considerations, we must have $\psi_2(1)=2$ and $\psi_3(1)=\psi_4(1)=6$. So the $\mathsf{S}_N$-character table is 
\begin{center}
\begin{tabular}{|c|c|c|c|c|c|}
\hline
$\mathsf{S}_N$&1&$K_2$&$K_3$&$K_4$\\
\hline
$\mathbbm{1}$&1&1&1&1\\
\hline
$\psi_2$&2&$-1$&2&2\\
\hline
$\psi_3$&6&0&$3c_1$&$3c_2$\\
\hline
$\psi_4$&6&0&$3c_2$&$3c_1$\\
\hline
\end{tabular}\\[5ex]
\end{center}

Similarly, if $\eta_i\in\mathrm{Irr}(\mathsf{S}_H)$ is such that $\inner{\eta_i,(\chi_i)_N}>0$, the $\mathsf{S}_H$-character table is 
\begin{center}
\begin{tabular}{|c|c|c|c|c|c|}
\hline
$\mathsf{S}_H$&1&$K_3$&$K_4$\\
\hline
$\mathbbm{1}$&1&1&1\\
\hline
$\eta_3$&2&$c_1$&$c_2$\\
\hline
$\eta_4$&2&$c_2$&$c_1$\\
\hline
\end{tabular}\\[5ex]
\end{center}

Since $H=\{1\}\cup K_3\cup K_4$, and since both $\psi_3$ and $\psi_4$ vanish on $K_2=N\setminus H$, it follows that $\mathsf{S}_N$ is a nontrivial $\ast$-product over $H$. Since $\norm{\mathsf{S}_N}=4=\norm{\mathsf{S}_H}+1$, we conclude that
\[\mathsf{S}_N=\mathsf{S}_H\ast\mathsf{M}(N/H).\]
Moreover, it follows that 
\[\mathsf{S}=\left(\mathsf{S}_H\ast\mathsf{M}(N/H)\right)\mathop{\Delta}\mathsf{S}_{G/H}.\]
\end{exam}

We conclude with the following consequences of Lemma~\ref{vanishingcontainment} and Theorem~\ref{introstarprod}.  The first shows that we may use the $\mathsf{S}$-character table to deduce whether or not $\mathsf{S}$ is a nontrivial $\ast$-product or a nontrivial $\Delta$-product in a systemic way.
\begin{thm}\label{pairs}
Let $\mathsf{S}$ be a supercharacter theory of $G$, and let $N>1$ be $\mathsf{S}$-normal. There exists an $\mathsf{S}$-normal subgroup $H$ containing $N$ for which $\mathsf{S}=\mathsf{S}_H\mathop{\Delta}\mathsf{S}_{G/N}$ if and only if ${\bf V}(\mathsf{S}\mid N)<G$.
\end{thm}
\begin{proof}
One direction is clear from Theorem~\ref{delta}. Now suppose that ${\bf V}(\mathsf{S}\mid N)<G$. If $N={\bf V}(\mathsf{S}\mid N)$, then $\mathsf{S}$ is an $\ast$-product over $N$ (a $\Delta$-product over $N$ and $N$) by Theorem~\ref{introstarprod}. Otherwise, we must have $N<{\bf V}(\mathsf{S}\mid N)$ by Lemma~\ref{vanishingcontainment}; hence $\mathsf{S}$ is a $\Delta$-product over $N$ and ${\bf V}(\mathsf{S}\mid N)$ by Theorem~\ref{delta}.
\end{proof}

We now obtain a similar result specifically for $\ast$-products. If $N\lhd_{\mathsf{S}}G$ and $N\nsubseteq[G,\mathsf{S}]$, it is easy to see that $\mathbf{V}(\mathsf{S}\mid N)=G$. So let us consider an $\mathsf{S}$-normal subgroup $N\le[G,\mathsf{S}]$. If $\mathbf{V}(\mathsf{S}\mid N)= N$, then $\mathsf{S}$ is a $\ast$-product over $N$ by Theorem~\ref{introstarprod}; otherwise $\mathbf{V}(\mathsf{S}\mid N)>N$  by Lemma~\ref{vanishingcontainment}. Thus we may consider the subgroup $\mathbf{V}^2(\mathsf{S}\mid N)=\mathbf{V}(\mathsf{S}\mid\mathbf{V}(\mathsf{S}\mid N))$. We may then define the ascending chain of subgroups $\mathbf{V}^i(\mathsf{S}\mid N)$, defined by $\mathbf{V}^i(\mathsf{S}\mid N)=\mathbf{V}(\mathsf{S}\mid\mathbf{V}^{i-1}(\mathsf{S}\mid N))$ for $i\ge 2$.

\begin{lem}
Let $\mathsf{S}$ be a supercharacter theory of $G$, and let $N>1$ be $\mathsf{S}$-normal. Let $V$ be the terminal member of $\mathbf{V}^i(\mathsf{S}\mid N)$. Then $V<G$ if and only if there exists a proper $\mathsf{S}$-normal subgroup $H$ of $G$ containing $N$ for which $\mathsf{S}$ is a $\ast$-product over $H$.
\end{lem}
\begin{proof}
First suppose that $V<G$. Then there is some index $n$ for which we have $\mathbf{V}^n(\mathsf{S}\mid N)=\mathbf{V}^{n+1}(\mathsf{S}\mid N)=\mathbf{V}(\mathsf{S}\mid\mathbf{V}^n(\mathsf{S}\mid N))$. Then $\mathsf{S}$ is a $\ast$-product over $V$ and clearly $N\le V$.

Now suppose that there exists a proper $\mathsf{S}$-normal subgroup $H$ of $G$ containing $N$ for which $\mathsf{S}$ is a $\ast$-product over $H$. Then, since $\mathbf{V}(\mathsf{S}\mid H)=H$, we have $\mathbf{V}^i(\mathsf{S}\mid N)\le H$ for all $i$ by Lemma~\ref{mono}. In particular, $V\le H<G$. 
\end{proof}
\bibliographystyle{plain}
\bibliography{bio}
\end{document}